\documentclass{article}
\pdfoutput=1
\usepackage[margin=1in]{geometry} 
\usepackage{amsmath,amsthm,amssymb,amsfonts}

\title{Nontrivial Steenrod Squares on Prime, Hyperbolic and Satellite Knots} 
\date{}
\author{Holt Bodish\thanks{
                  University of Oregon}
       \thanks{This work was
                  supported by NSF grant number DMS-1810893}
        }

\usepackage[lite]{amsrefs}
\usepackage{graphicx}
\usepackage{amssymb}
\usepackage{amsmath}
\usepackage{mathtools}
\usepackage{xypic}
\usepackage[all,cmtip]{xy}
\usepackage{mathrsfs}
\usepackage{tikz-cd} 

\usepackage{enumerate}

\DeclareMathOperator{\Kh}{Kh}
\DeclareMathOperator{\Ker}{Ker}
\DeclareMathOperator{\Sq}{Sq}
\DeclareMathOperator{\Image}{Im}

\tikzset{
  curarrow/.style={
  rounded corners=8pt,
  execute at begin to={every node/.style={fill=red}},
    to path={-- ([xshift=-50pt]\tikztostart.center)
    |- (#1) node[fill=white] {$\scriptstyle d_*$}
    -| ([xshift=50pt]\tikztotarget.center)
    -- (\tikztotarget)}
    }
}

\newcommand{\Z}{\mathbb{Z}}

\newcommand{\F}{\mathbb{F}}

\usepackage{graphicx}
\usepackage{cite}

\newtheorem{theorem}{Theorem}[section]
\newtheorem{lemma}[theorem]{Lemma}
\newtheorem{corollary}[theorem]{Corollary}

\newtheorem{definition}[theorem]{Definition}

\begin{document}

\newpage
\maketitle

\begin{abstract}
        We use the work of Lipshitz-Sarkar \cites{Sinvariant,MR3966803} and Lawson-Lipshitz-Sarkar \cite{TylerLawsonRobertLipshitz} as well as Wilson and Levine-Zemke \cite{MR3122052,MR4041014} to prove that there are prime knots, in fact hyperbolic and prime satellite knots, with arbitrarily high Steenrod squares on their (reduced and unreduced) Khovanov homology. 

\end{abstract}

\section{Introduction}
\label{sec: Introduction}
In 2014, Lipshitz and Sarkar introduced a stable homotopy refinement of Khovanov homology \cite{MR3230817}. For each fixed $j$ it takes the form of a suspension spectrum $\mathcal{X}^j$. The cohomology $H^*(\mathcal{X}^j)$ of this spectrum is isomorphic to the Khovanov homology $\Kh^{*,j}$. In subsequent work (e.g. \cite{MR3252965}) they used this refinement to define stable cohomology operations on Khovanov homology. This lead to a refinement of Rasmussen's $s$-invariant for each nontrivial cohomology operation, and in particular for the Steenrod squares \cite{MR3252965}. In this short note we offer a solution to the following question posed in Lipshitz-Sarkar \cite[Question 3]{MR3966803}: Are there prime knots with arbitrarily high Steenrod squares on their Khovanov homology? Explicitly, we prove the following theorem:

\begin{theorem}
\label{PrimeSq}
Given any $n$, there exists a prime knot $P_n$ so that the operation
$$\Sq^n: \widetilde{\Kh}^{i,j}(P_n) \to \widetilde{\Kh}^{i+n,j}(P_n)$$ is nontrivial for some $(i,j)$. Here $\widetilde{\Kh}$ denotes reduced Khovanov homology. \\

\end{theorem}

\begin{corollary}
\label{PrimeSqUn}
 Given any $n$, there exists a prime knot $P_n$ so that the operation
$$\Sq^n: \Kh^{i,j}(P_n) \to \Kh^{i+n,j}(P_n)$$ is nontrivial, on unreduced Khovanov homology, for some $(i,j)$. 
\end{corollary}

In fact a stronger version of Theorem \ref{PrimeSq} is true: 

\begin{theorem}
\label{Hyper}
Given any $n$, there exists a hyperbolic knot $H_n$ so that the operation
$$\Sq^n : \widetilde{\Kh}^{i,j}(H_n) \to \widetilde{\Kh}^{i+n,j}(H_n)$$ is nontrivial for some $(i,j)$. Here $\widetilde{\Kh}$ denotes reduced Khovanov homology. \\
\end{theorem}

\begin{theorem}
\label{Satellite}

Given any $n$, there exists a prime satellite knot $S_n$ so that the operation
$$\Sq^n : \widetilde{\Kh}^{i,j}(S_n) \to \widetilde{\Kh}^{i+n,j}(S_n)$$ is nontrivial for some $(i,j)$. Here $\widetilde{\Kh}$ denotes reduced Khovanov homology. \\
\end{theorem}

Our technique for proving all of the above theorems is to find a ribbon concordance from any given knot $K$ to a prime, hyperbolic or satellite knot, then appeal to the following generalization to reduced Khovanov homology of a theorem of Wilson and Levine-Zemke (for the original statement see \cites{MR3122052,MR4041014} or Theorem \ref{Ribboninjective} below).
\begin{theorem}
\label{RibbonInj}
Suppose $C$ is a ribbon concordance between knots $K$ and $K'$. Then the induced map $F_C : \widetilde{\Kh}(K) \to \widetilde{\Kh}(K')$ is injective. 

\end{theorem}
Recall that any prime knot $K$ is either a hyperbolic knot, a satellite knot, or a torus knot. With this in mind, Theorems \ref{PrimeSq}---\ref{Satellite} suggest the following question: \\

\noindent\textbf{Question}: For any given $n$, is there a torus knot $T_n$ so that $\Sq^n : \widetilde{\Kh}^{i,j}(T_n) \to \widetilde{\Kh}^{i+n,j}(T_n)$ is nontrivial for some $(i,j)$?\\

The organization of the paper is the following. In Section \ref{KhandRibbon}, we review the results of Wilson and Levine and Zemke  \cite{MR3122052,MR4041014} showing that ribbon concordances induce split injections on Khovanov homology. In Section \ref{Basepoint}, we prove the analogue of this theorem for reduced Khovanov homology. In Section \ref{KnotsP},  we show that any knot is ribbon concordant to a prime knot, following the arguments in \cites{MR621991,MR542689}. In Section \ref{Steenrod}, we collect various results about the naturality of Steenrod squares with respect to births, Reidemeister moves and saddle maps and the behavior of the Khovanov stable homotopy type under connected sums. In Section \ref{Proof} we show that the nontriviality of Steenrod squares on composite knots constructed by Lipshitz-Sarkar \cite[Corollary 1.4]{TylerLawsonRobertLipshitz} and \cite[Corollary 3.1]{MR3966803} propagates to the nontriviality of Steenrod squares on the Khovanov homology of prime knots. In Section \ref{Hyp} we prove Theorem \ref{Hyper}  using results of Silver and Whitten \cite{Silver}. In Section \ref{SatelliteKnots} we prove Theorem \ref{Satellite} using results of Livingston \cite{MR625818}.\\

\noindent\textbf{Acknowledgements} The author would like to thank his advisor Robert Lipshitz, as well as Danny Ruberman for pointing out Theorem~\ref{Hyper}, Chuck Livingston for pointing out his construction of ribbon concordances to prime knots that leads to Theorem \ref{Satellite}, and the anonymous referee for their careful reading that improved the exposition.  

\section{Khovanov Homology and Ribbon Concordances}\label{KhandRibbon}

In this section we review the behavior of Khovanov homology under ribbon concordances. Unless explicitly stated otherwise, throughout this paper we write $\Kh(K)$ to mean $\Kh(K;\F_2)$. 

\begin{definition}
let $K_0$ and $K_1$ be links in $S^3$. A \emph{concordance} from $K_0$ to $K_1$ is a smoothly embedded cylinder in $[0,1] \times S^3$ with boundary $-(\{0\} \times K_0) \cup (\{1\} \times K_1)$. A concordance $C$ is said to be \emph{ribbon} if $C$ has only index $0$ and $1$ critical points with respect to the projection $[0,1] \times S^3 \to [0,1]$. 
\end{definition}

Throughout this paper, we will use the notation $\overline{C}$ to denote the ribbon concordance $C$ upside-down.

\begin{theorem}\label{Ribboninjective}\cites{MR3122052,MR4041014} If $C$ is a ribbon concordance from $K_0$ to $K_1$, then the induced map $$\Kh(C): \Kh(K_0) \to \Kh(K_1)$$ is injective, with left inverse $\Kh(\overline{C})$. In particular, for any bigrading $(i,j)$ the group $\Kh^{i,j}(K_0)$ is a direct summand of the group $\Kh^{i,j}(K_1)$. 

\end{theorem}

The proof of this theorem involves decomposing the cobordism $D:= \overline{C} \circ C$ as the disjoint union of the identity cobordism (a cylinder) and sphere components joined to the cylinder by tubes (formed from the ribbons and their duals). For details, see \cite{MR4041014} or \cite{MR3122052}.
In the next section, we present an analogue of Theorem \ref{Ribboninjective} for reduced Khovanov homology, after reviewing the necessary definitions.

\section{The Base-point action and Reduced Khovanov Homology}\label{Basepoint}

We begin with the definition of the base-point action on Khovanov homology. For grading conventions, see \cite{Shuma}.

\begin{definition}

Fix a diagram of the knot $K$ and pick a base-point $q \in K$ not on any of the crossings. Then we make the Khovanov complex $C_{\Kh}(K)$ of $K$ into a module over $\F_2[X]/X^2$ as follows. Generators of the chain groups are complete resolutions of $K$ and a choice of $1$ or $X$ for each component of the complete resolution. Multiplication by $X$ is zero if the generator labels the circle containing $q$ with an $X$ and if the generator labels the circle containing $q$ by $1$ it changes the label of the circle to $X$. With our grading conventions (see \cite{Shuma}), multiplication by $X$ has bidegree $(0,-2)$. That is

$$X: \Kh^{i,j}(K) \to \Kh^{i,j-2}(K).$$
\end{definition}

\begin{definition} Let $\F$ be the $\F_2[X]/X^2$ module $\F_2$ where $X$ acts trivially. Then define

$$\widetilde{C}_{\Kh}(K):=C_{\Kh}(K) \otimes_{\F_2[X]/X^2} \F.$$ 

The homology of the complex $\widetilde{C}_{\Kh}(K)$ is called \emph{reduced Khovanov homology} and denoted $\widetilde{\Kh}(K)$. 
\end{definition}

\begin{theorem}\cite[Corollaries 3.2.B and 3.2.C]{Shuma}\label{Ex-act} The action of $X$ on $C_{\Kh}(K)$ commutes with the Khovanov differential, so induces a map (also called $X$) on homology. Further, 

\begin{enumerate}

\item The following sequence is exact: 
$$\cdots \xrightarrow{X} \Kh^{i,j+2}(K) \xrightarrow{X} \Kh^{i,j}(K) \xrightarrow{X} \Kh^{i,j-2}(K) \xrightarrow{X} \cdots$$

\item The reduced Khovanov homology over $\F_2$ is isomorphic to the kernel of $X$ (which is the image of $X$ by part $1$), and we have the direct sum decomposition

$$\Kh^{i,j}(K) \cong \widetilde{\Kh}^{i,j-1}(K) \oplus \widetilde{\Kh}^{i,j+1}(K).$$

\end{enumerate}
\end{theorem}

With these preliminaries in mind, we prove Theorem \ref{RibbonInj} from the introduction.

\begin{proof}[Proof of Theorem \ref{RibbonInj}]

By Theorem \ref{Ribboninjective} we know that the map $F_C: \Kh(K_0) \to \Kh(K_1)$ is a split injection with left inverse $F_{\overline{C}}$. By Theorem \ref{Ex-act}, for $a \in \{0,1\}$, $$\widetilde{\Kh}(K_a) \cong \Ker(X: \Kh(K_a) \to \Kh(K_a))\cong \Image(X: \Kh(K_a) \to \Kh(K_a)).$$ Therefore, it is enough to show that the map $F_C$ is a $\F_2[X]/X^2$ module map. Indeed, then $F_C|_{\Ker}$ maps  $\Ker(X: \Kh(K_0) \to \Kh(K_0))$ to $\Ker(X: \Kh(K_1) \to \Kh(K_1))$ and $F_{\overline{C}}|_{\Ker}$ maps $\Ker(X: \Kh(K_1) \to \Kh(K_1))$ to $\Ker(X: \Kh(K_0) \to \Kh(K_0)).$ Further, $F_{\overline{C}}|_{\Ker}\circ F_C|_{\Ker}=id|_{\Ker}$. Therefore $F_C|_{\Ker}$ is a split injection. \\

Now, any cobordism can be decomposed into births (0-handles) and saddle moves (1-handle attachments) and deaths (2-handles). So, to show that the maps induced on Khovanov homology by cobordisms respect the $X$ action, it suffices to verify the following. 

\begin{enumerate}
\item Births and deaths respect the module structure with respect to a base-point not on the circle dying or being born.  
\item The isomorphisms of Khovanov homology associated to Reidemeister moves respect the module structure. 
\item The maps associated with saddles respect the module structure.

\end{enumerate}

 Item $1$ is clear from the definition of the $X$ action, provided we chose a base-point on the original knot diagram, away from where the births and deaths occur.\\

Item $2$ follows from Proposition $2.2$ of Hedden-Ni \cite{MR3190305}. Evidently, the homotopy equivalences induced from Reidemeister moves commute with the $X$ action if the Reidemeister moves does not involve a strand moving across a base-point. Therefore it suffices to show that moving a strand across the base-point does not change the action of $X$ on homology. This follows by writing down an explicit chain homotopy between the different base-point actions associated with choosing two marked points, on the same component, on opposite sides of a crossing. These homotopy equivalences appear in \cite[Lemma 2.3]{MR3190305}.  \\

Item $3$ reduces to a local calculation in a complete resolution. Either the saddle cobordism merges two components, or splits one component into two. In either case, it is easy to check that the maps involved commute with the $X$ action. \qedhere

\end{proof}

\section{Knots and Prime Tangles}\label{KnotsP}

The main theorem of this section is the following:

\begin{theorem}\cites{MR621991, MR542689}\label{Prime}
Any knot is ribbon concordant to a prime knot. 
\end{theorem}

The proof of this theorem is standard and is well explained elsewhere in the literature. We include a review of the techniques used in the proof for the convenience of the reader and to introduce some notation. We begin with a definition and a convention \cites{MR621991, MR542689, Bleiler}.

\begin{definition}

A (4-ended) \emph{tangle} with no closed components is an embedding of $[0,1] \sqcup [0,1]$ into $B^3$ so that $\{0,1\} \cup \{0,1\}$ map to $S^2=\partial B^3$. We specify a tangle by a diagram, see Figure \ref{Clasp}. We denote such a tangle by $(B,T)$ or just $T$. A tangle $(B,T)$ is \emph{prime} if both of the following conditions hold:
\begin{enumerate}
\item Any $2$-sphere embedded in $B$ that intersects the knot transversely at two points bounds on one side a three ball $A$ so that $A \cap T$ is homeomorphic to the standard ball arc pair $(D^2 \times [0,1], 0 \times [0,1])$.  

\item $(B,T)$ is not a rational tangle. Equivalently, $(B,T)$ does not contain any separating disks. 
\end{enumerate}

\end{definition}

One motivation for the name prime tangle and illustration of their use is indicated by the following:

\begin{theorem}\cite[Lemmas 1, 2]{MR621991} \label{Sum}
The sum of two prime tangles is a prime knot. The partial sum of two prime tangles is a prime tangle. 
\end{theorem}

For the proof, see \cite{MR621991}. In this paper, we use the notation $+_p$ for the partial sum of two tangles and the notation $T_1+T_2$ for the sum of two tangles. These operations depend on a choice of which endpoints are identified. In the present work, the operations $+_p$ and $+$ mean the operations in Figures \ref{Partial} and  \ref{fig:prob1_6_12} respectively. For our purposes, we make this explicit as follows. Let NW, NE, SW, SE denote the northwest, northeast, etc corners of the diagram of a tangle $T$. Then $T_1 +_p T_2$ means the tangle formed by joining the NE and SE corners of $T_1$ to the NW and SW corners of $T_2$ respectively by unknotted arcs. Further, $T_1 + T_2$ means the tangle formed from $T_1$ and $T_2$ by joining the $NE$ corner of $T_1$ with the SE corner of $T_2$, the SE corner of $T_1$ with the NE corner of $T_2$, etc. See Figure \ref{fig:prob1_6_12}, which shows $(T_1+_p T_2) + Cl$. Note that, even though we use the $+$ sign to denote the tangle sum operation, it is usually not commutative.

\begin{figure}
    \centering
    \begin{minipage}{.5\textwidth}
      \centering
        \includegraphics[scale=.09]{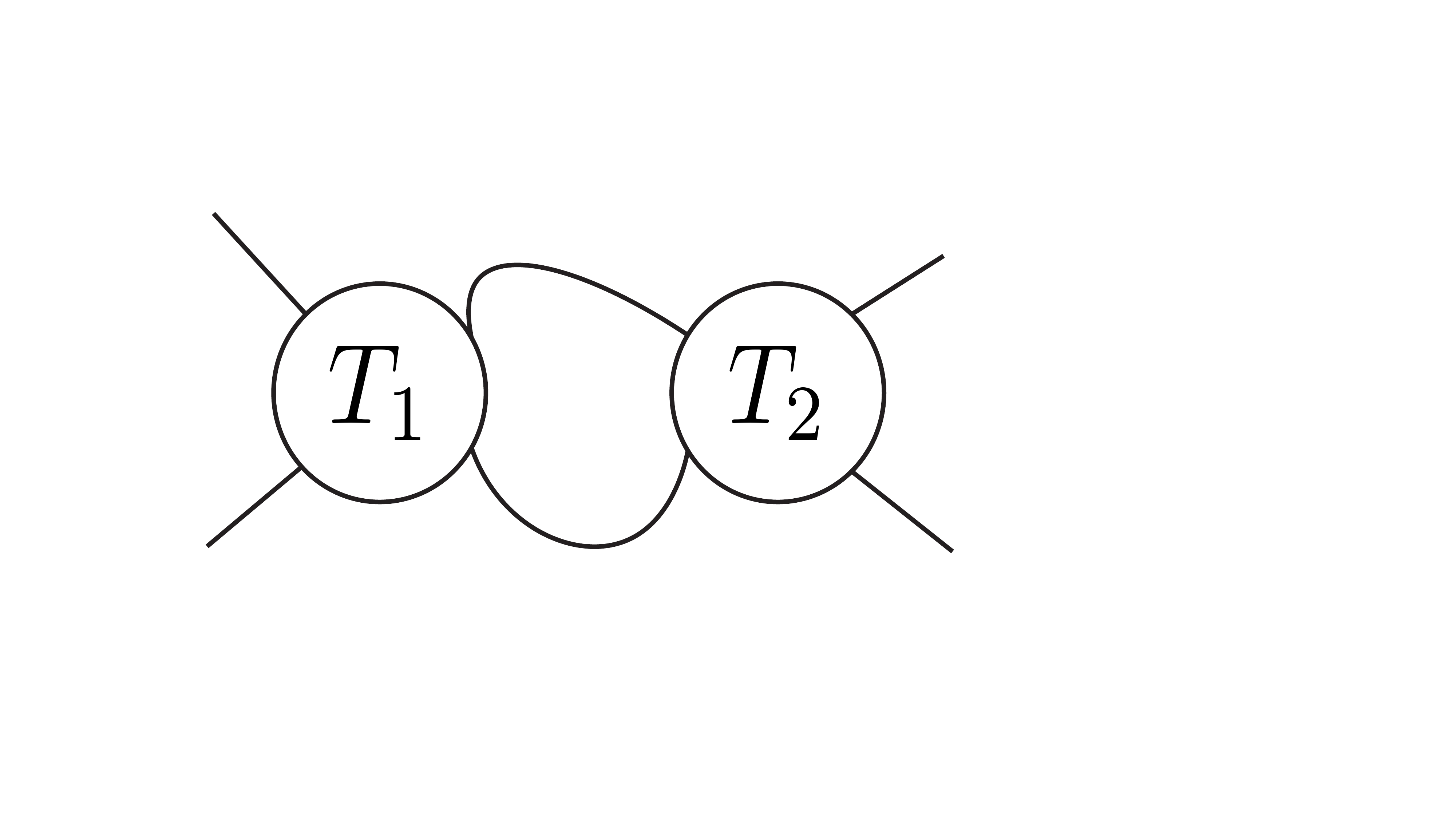}
        \caption{$T_1 +_p T_2$}
        \label{Partial}
        \label{fig:prob1_6_2}
    \end{minipage}%
    \begin{minipage}{0.5\textwidth}
        \centering
        \includegraphics[scale=.09]{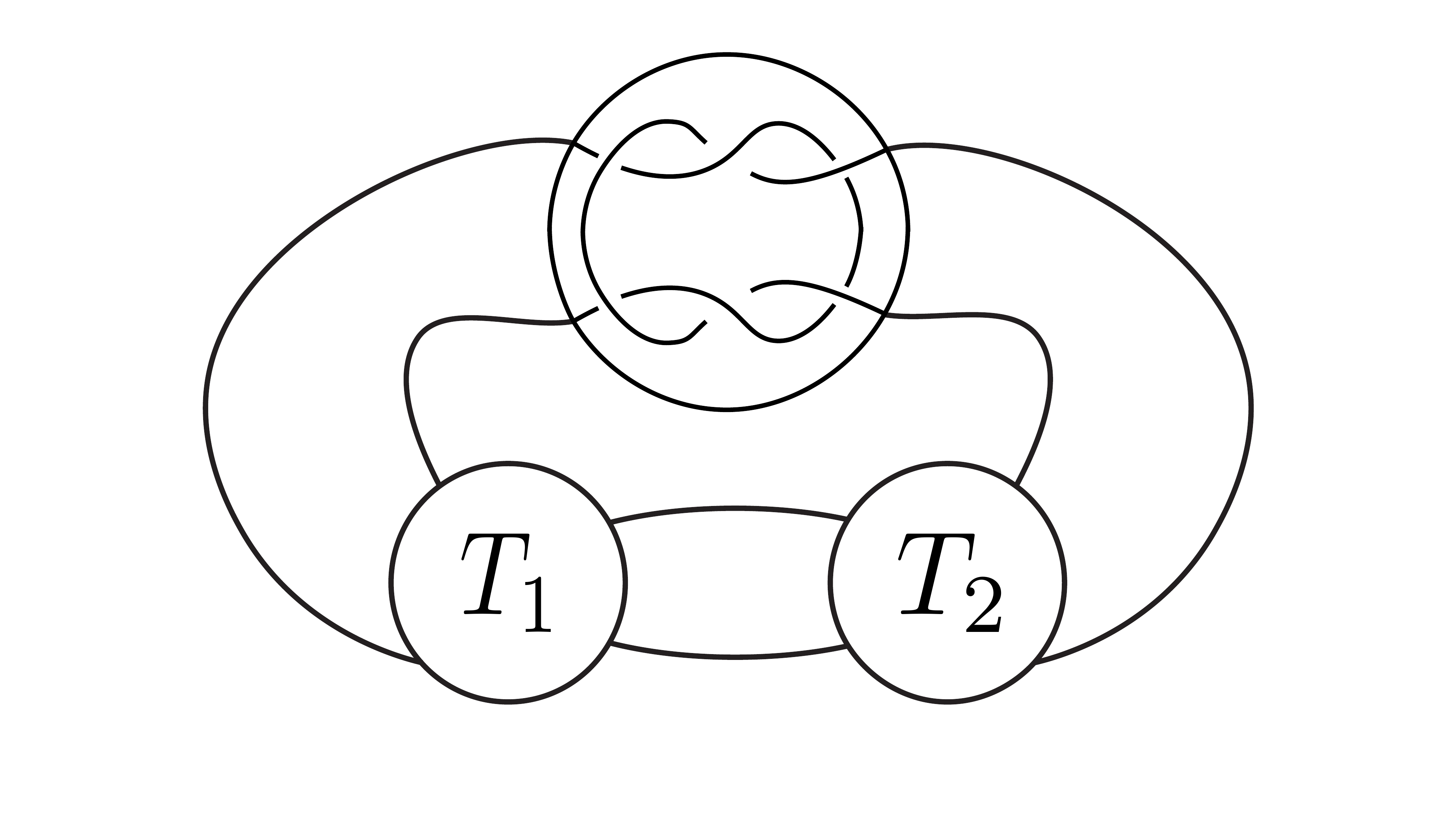}
        \caption{$(T_1+_pT_2)+Cl$}
        \label{fig:prob1_6_12}
    \end{minipage}
\end{figure}

\begin{lemma}[See \cite{MR542689},  \cite{Bleiler}] 
\label{StandardPrime}
For any nontrivial knot $K$ in $S^3$ there is an embedded $S^2$ meeting $K$ transversely in four points separating $S^3$ into two three balls $A$ and $B$ so that 

\begin{enumerate}
\item $(A, A \cap K)$ is a trivial two-stranded tangle (so homeomorphic, as pairs, to $(D^2 \times I, \{(-1/2,0)\} \times I \cup \{(1/2,0)\} \times I))$, and

\item $(B,B \cap K)$ is a prime tangle.

\end{enumerate}
\end{lemma}

\begin{lemma}\label{Claspy}
The clasp tangle $Cl$ is a prime tangle. 

\begin{proof}
Since each of the individual strings that compose the clasp tangle are unknotted, condition $1$ in the definition of a prime tangle is automatically satisfied. We just need to verify that the clasp is not a rational tangle. Suppose for the sake of contradiction that it is. Recall that a knot built out of two rational tangles is a two-bridge knot. It is a classical fact (originally proved by Schubert, see J. Schultens \cite{MR2018265} for a modern proof) that the bridge number of a knot, $b(K)$, satisfies $b(K \# K')=b(K)+b(K')-1$. Further, the only knot with bridge number $1$ is the unknot. These two facts together imply that two-bridge knots are prime. However, the numerator closure of the clasp tangle is clearly a connected sum $3_1 \# m(3_1)$. \qedhere

\end{proof}

\begin{figure}[!htb]
    \centering
    \begin{minipage}{.5\textwidth}
        \centering
      \includegraphics[scale=.09]{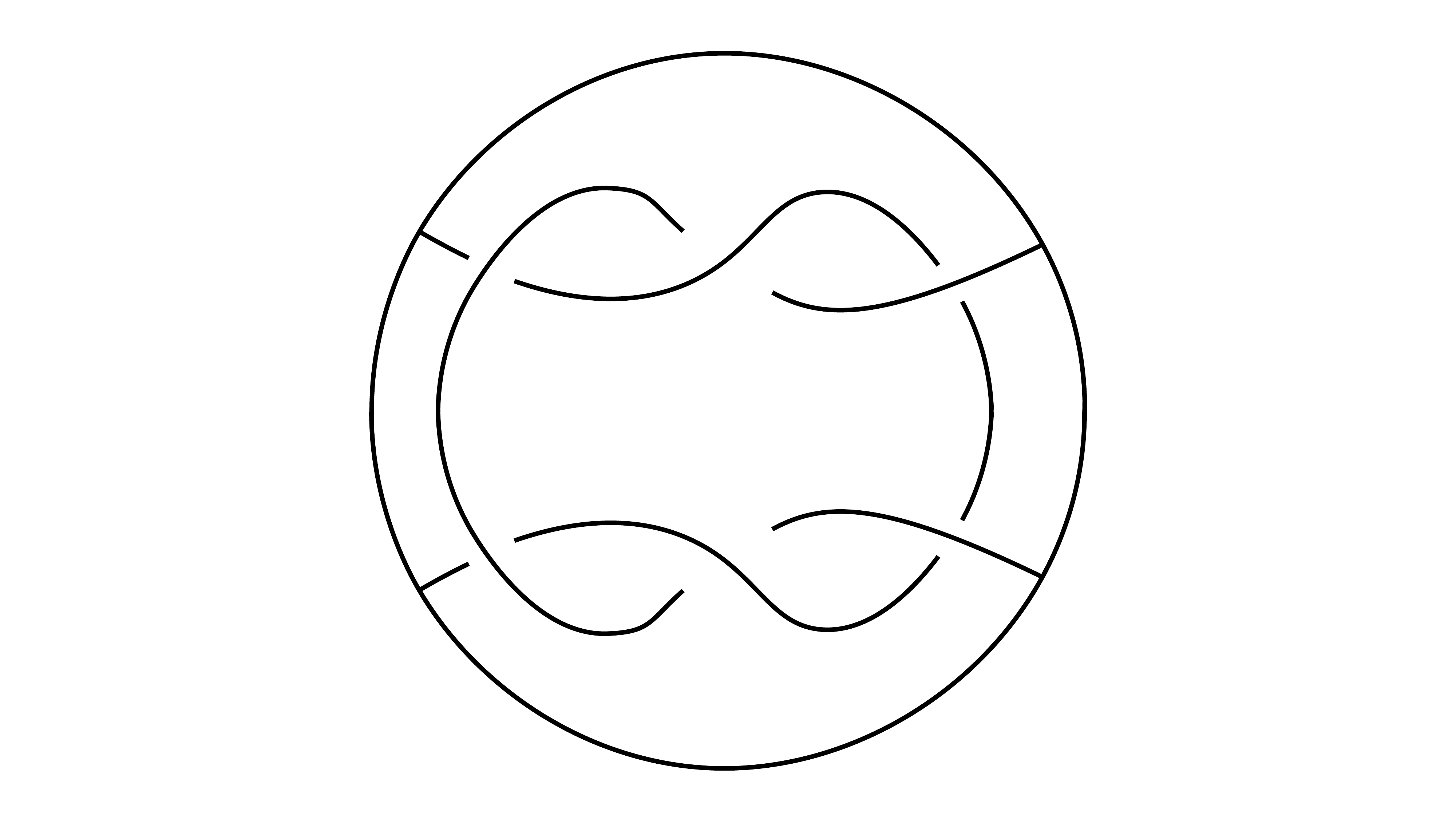}
        \caption{The clasp tangle $Cl$}
        \label{Clasp}
    \end{minipage}%
    \begin{minipage}{0.5\textwidth}
        \centering
        \includegraphics[scale=.09]{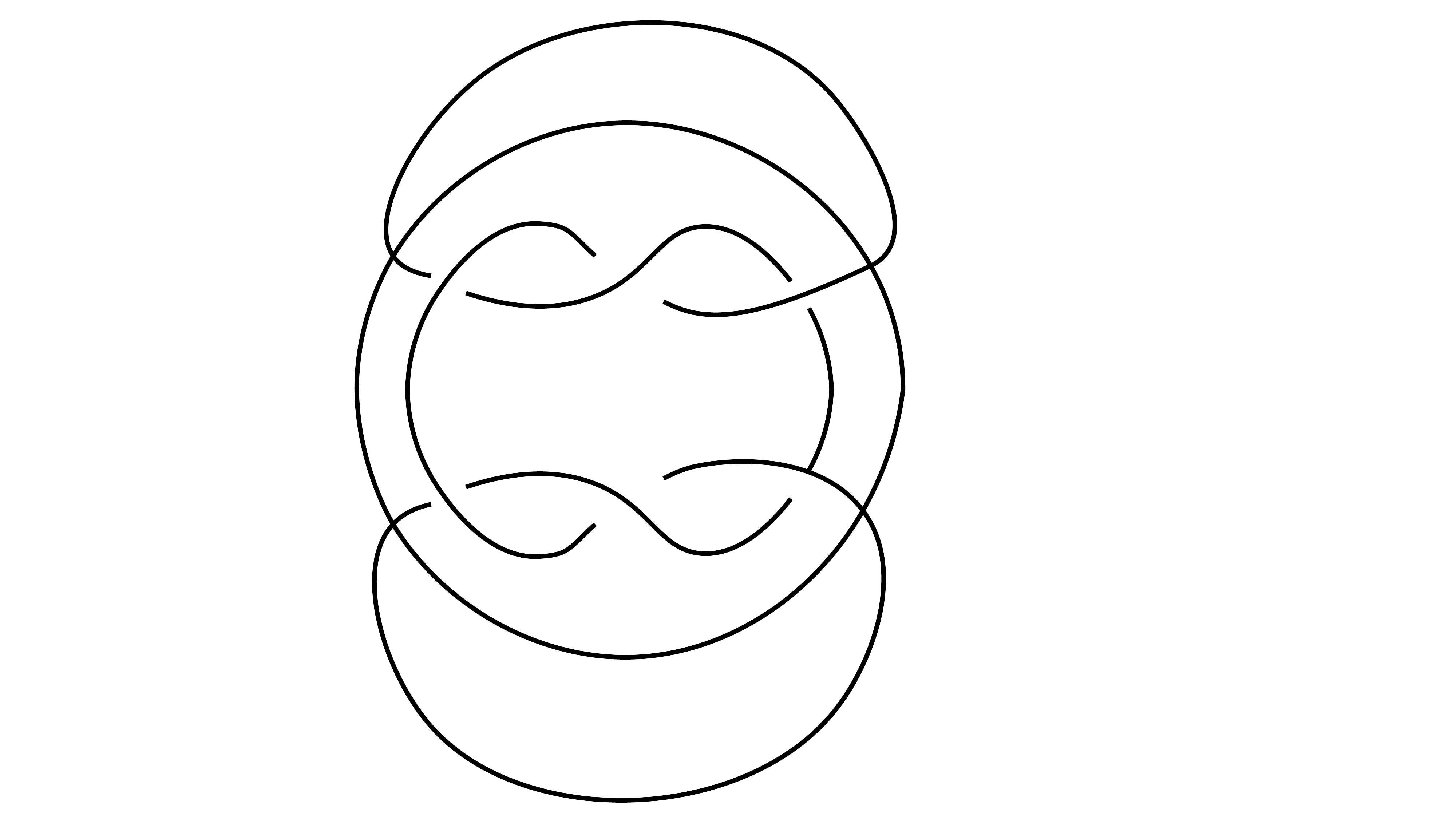}
        \caption{The numerator closure of the clasp tangle}
        \label{fig:prob1_6_1}
    \end{minipage}
\end{figure}

\end{lemma}

\begin{proof}[Proof of Theorem \ref{Prime}] Since any knot can be decomposed as a connected sum of prime knots, and connected sum is compatible with concordance, it suffices to prove the result for a knot $K=K_1 \# K_2$ where $K_i$ are prime. By Lemma \ref{StandardPrime}, we can find two disjoint three balls $B_1$ and $B_2$ so that $T_i=B_i \cap K_i$ is a prime tangle and $(S^3 \setminus B_i) \cap K_i$ is an untangle. Now, consider $T_1 +_p T_2$. This tangle is prime by Theorem \ref{Sum}. The denominator closure of the resulting tangle is $K_1 \# K_2$. The tangle sum $(T_1 +_p T_2)+Cl$ is then a prime knot by Theorem \ref{Sum}. The ribbon concordance, shown in Figure \ref{Denominator}---Figure \ref{Final}, between $K_1 \# K_2$ and $(T_1 +_p T_2)+Cl$  establishes the result.\qedhere

\end{proof}

\begin{figure}[!htb]
    \centering
    \begin{minipage}{.5\textwidth}
        \centering
      \includegraphics[scale=.09]{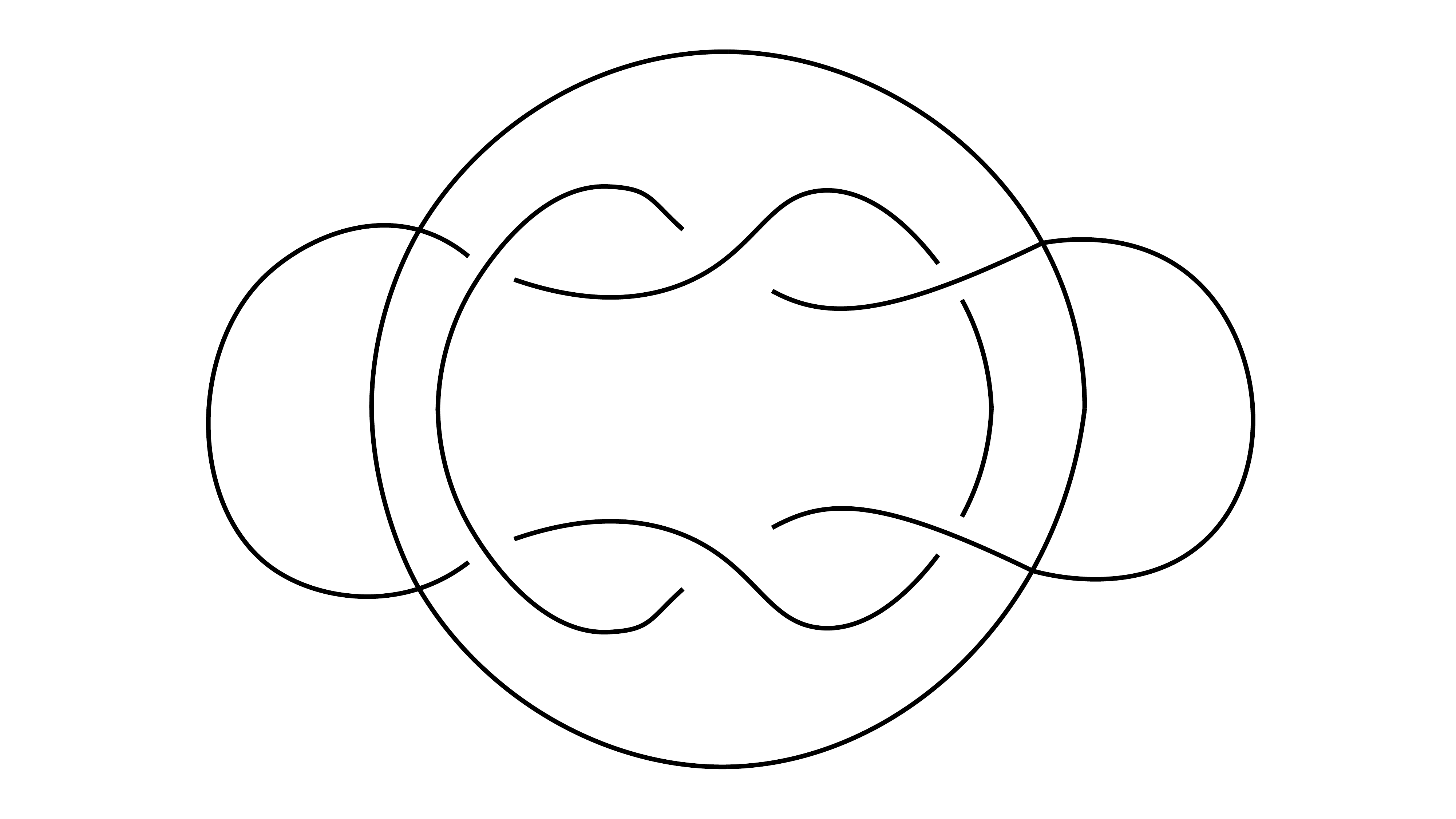}
        \caption{Denominator closure of the clasp tangle}
        \label{fig:prob1_6_2}
    \end{minipage}%
    \begin{minipage}{0.5\textwidth}
        \centering
        \includegraphics[scale=.09]{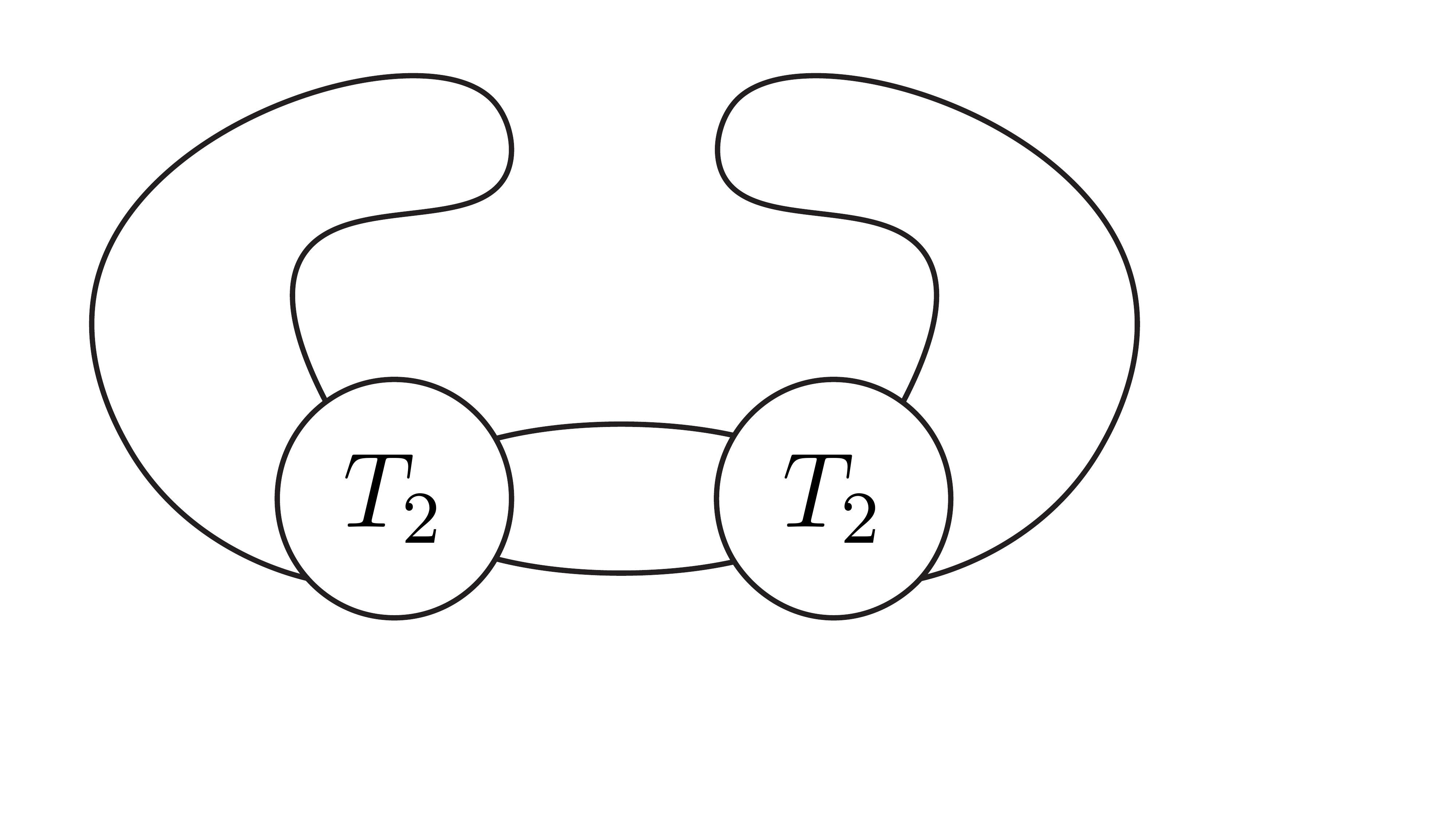}
        \caption{Denominator closure of $T_2 +_p T_2$}
        \label{fig:prob1_6_1}
    \end{minipage}
\end{figure}

\section{Steenrod Operations and Stable Homotopy Type}\label{Steenrod}

In this section we review, in bare bones fashion, the necessary facts about Khovanov stable homotopy type needed in establishing Theorem \ref{PrimeSq}.

We begin with a theorem, which explains how the Khovanov stable homotopy type behaves under the operation of connected sum. Throughout this section, let $L$ denote a link of one or more components. 

\begin{theorem} \cite[Theorem 2]{TylerLawsonRobertLipshitz}\label{connected}

$$\widetilde{\mathcal{X}}^j_{\Kh}(L_1 \# L_2) \simeq \bigvee_{j_1+j_2=j} \widetilde{\mathcal{X}}^{j_1}(L_1) \wedge \widetilde{\mathcal{X}}^{j_2}(L_2).$$

\end{theorem}

Next, we recall the precise naturality statement enjoyed by stable cohomology operations.

\begin{theorem} \cite[Theorem 4]{Sinvariant}  \label{Cobordism} Let $S$ be a smooth cobordism in $[0,1] \times S^3$ from $L_1$ to $L_2$, and  let $F_S: \Kh^{*,*}(L_1) \to \Kh^{*,*+\chi(S)}(L_2)$ be the map associated to $S$. Let $\alpha : \widetilde{H}^*(\cdot;\F) \to \widetilde{H}^{*+n}(\cdot;\F)$ be a stable cohomology operation. Then the following diagram commutes up to sign:

\[
\begin{tikzcd}
\Kh^{i,j}(L_1;\F) \arrow{r}{\alpha} \arrow[swap]{d}{F_S} & \Kh^{i+n,j}(L_1;\F) \arrow{d}{F_S} \\
\Kh^{i,j+\chi(S)}(L_2;\F)  \arrow{r}{\alpha} & \Kh^{i+n,j+\chi(S)}(L_2;\F).
\end{tikzcd}
\]

\end{theorem}

Then, for $\alpha$ a stable cohomology operation, the following diagram commutes:

\[
\begin{tikzcd}
\Kh(U \sqcup K;\F_2) \arrow{r}{\alpha} \arrow[swap]{d}{\cong} & \Kh(U \sqcup K;\F_2) \arrow{d}{\cong} \\
  \F_2[X]/X^2 \otimes \Kh(K;\F_2)  \arrow{r}{Id \otimes \alpha} \arrow[swap]{d}{m} &   \F_2[X]/X^2 \otimes \Kh(K;\F_2) \arrow{d}{m} \\
\Kh(K;\F_2)  \arrow{r}{\alpha} & \Kh(K;\F_2).
\end{tikzcd}
\]
The bottom square commutes by Theorem $9$, since the $X$ action on $\Kh$ can also be viewed as induced from a merge cobordism $U\sqcup K \to K$ where the unknot is placed near the basepoint. The top square commutes since the Khovanov spectrum of the unknot is homotopy equivalent to a wedge of two $S^0$'s in grading $-1$ and $1$. This homotopy equivalence induces the map in cohomology that identifies $  \F_2[X]/X^2 \otimes \Kh(K;\F_2)$ with $\Kh(K;\F_2) \oplus \Kh(K;\F_2)$ with appropriate grading shifts.

Commutativity of the above diagram is the statement that any stable cohomology operation is a map of $\F_2[X]/X^2$ modules. It follows that the analogous diagram to the one in Theorem \ref{Cobordism}, with Khovanov homology replaced by reduced Khovanov homology commutes,  commutes. 

\begin{lemma}\cite[Corollary 1.4]{TylerLawsonRobertLipshitz}\label{SQ} For any $n$ there is a knot $K_n$ so that the operations

$$\Sq^n: \widetilde{\Kh}^{i,j}(K_n) \to \widetilde{\Kh}^{i+n,j}(K_n)$$ and

$$\Sq^n: \Kh^{i,j}(K_n) \to \Kh^{i+n,j}(K_n)$$
are nontrivial, for some $(i,j)$. 
\end{lemma}
 
 For the proof, see \cite[Proof of Corollary 1.4, Page 67]{TylerLawsonRobertLipshitz}. They find, for the knot $K=15^n_{41127}$, a class $\alpha \in \widetilde{\Kh}^{-1,0}(K; \F_2)$ so that $\Sq^1(\alpha) \neq 0 \in \widetilde{\Kh}^{0,0}(K;\F_2)$ and $\Sq^i(\alpha) =0$ for $i >1$.  Then, letting $K_n=K \# K \# \cdots \# K$, the Cartan formula and Theorem \ref{connected} give the result. \\






Since the knot $K_n$ in the above theorem is the knot $K$ connect summed with itself $n$ times, we can view $K_n$ as the denominator closure of the partial tangle sum $K +_p \cdots +_p K$ (see Figure \ref{Denominator}). 

\begin{figure}[!htb]
    \centering
    \begin{minipage}{.5\textwidth}
        \centering
        \includegraphics[scale=.09]{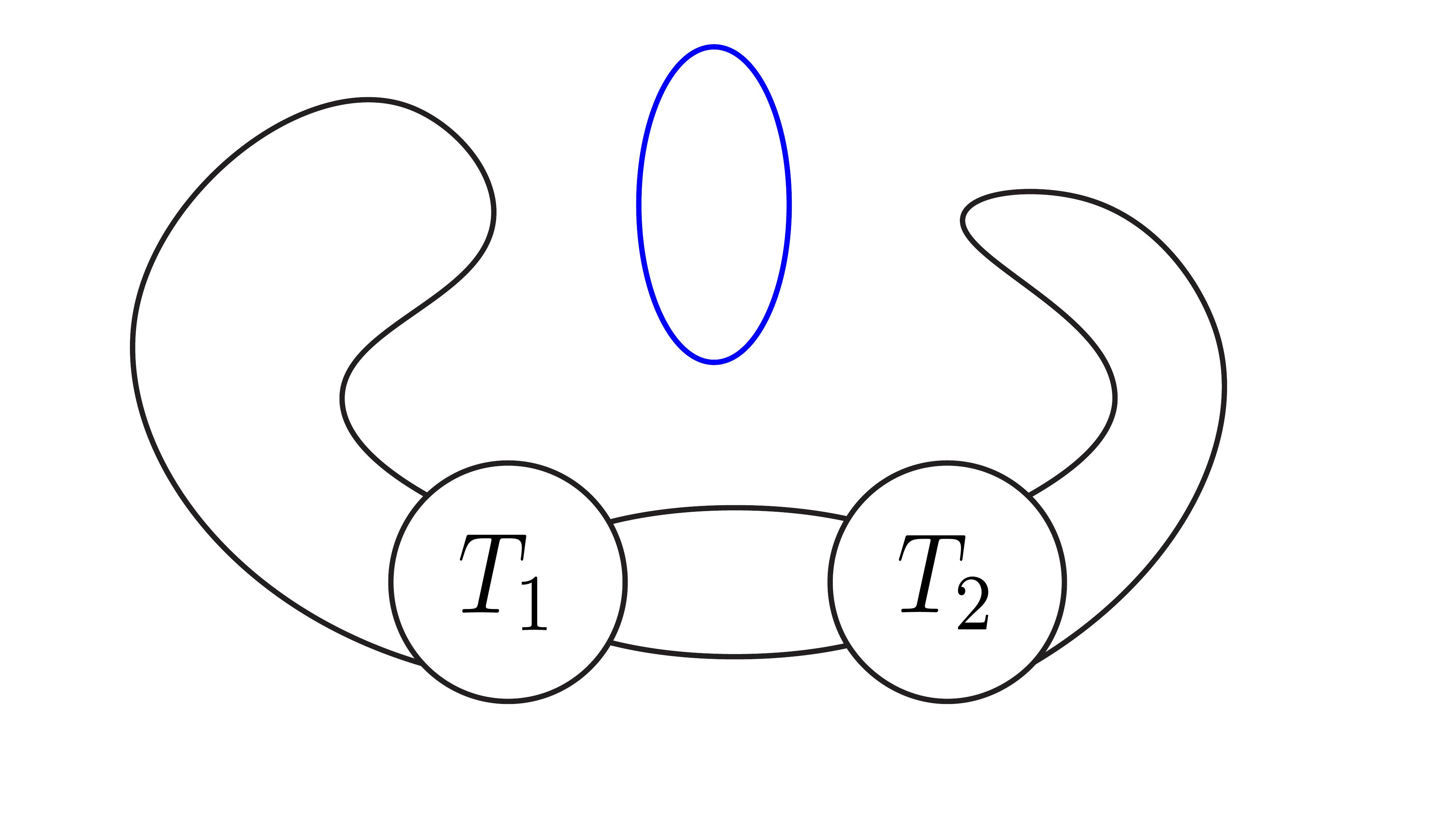}
        \caption{$K_1 \# K_2 \bigsqcup \text{Unknot}$}
        \label{Denominator}
        \label{fig:prob1_6_2}
    \end{minipage}%
    \begin{minipage}{0.5\textwidth}
        \centering
        \includegraphics[scale=.09]{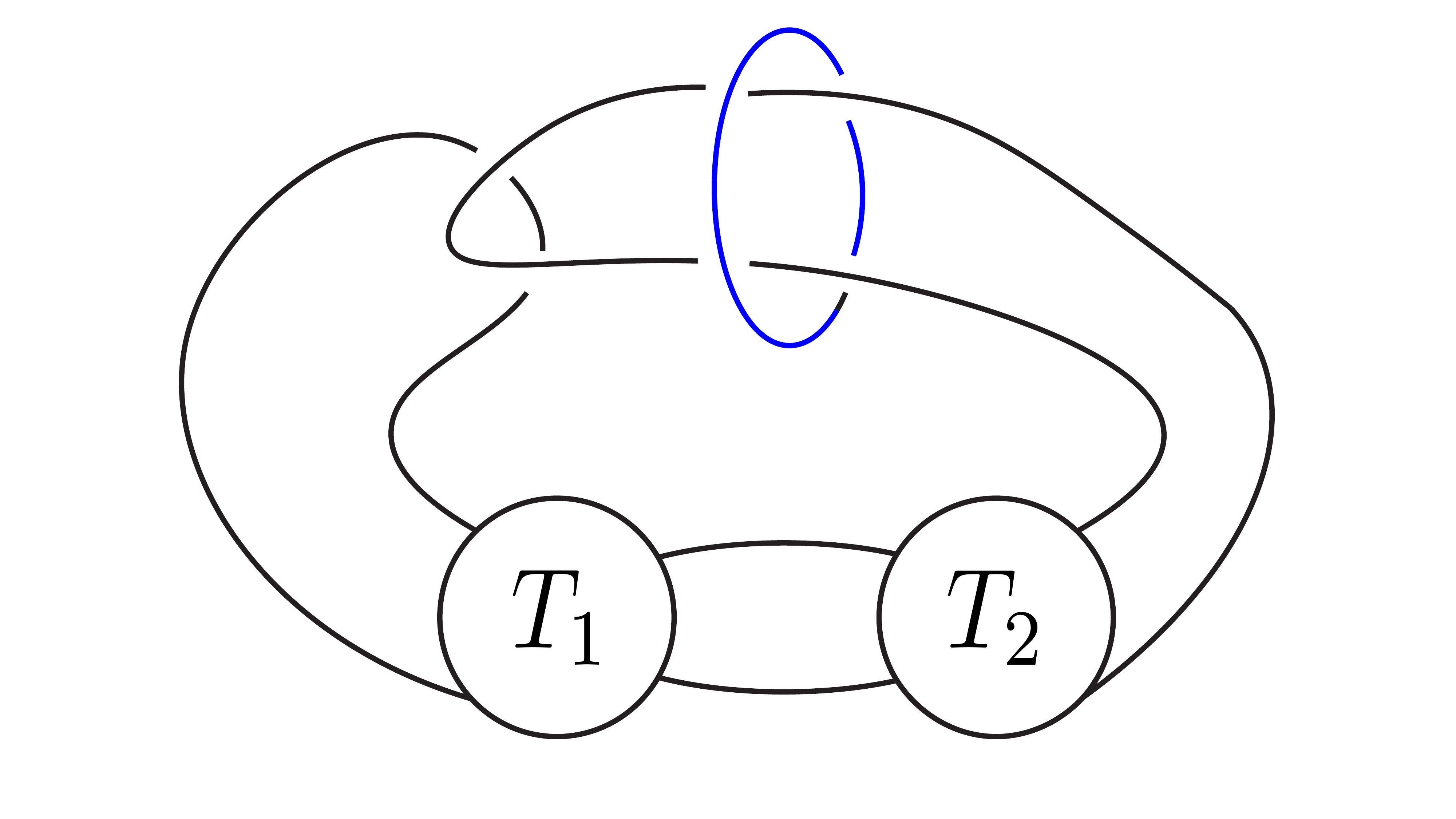}
        \caption{$K_1 \# K_2 \sqcup \text{Unknot}$ after isotopy}
        \label{fig:prob1_6_1}
    \end{minipage}
     \begin{minipage}{.5\textwidth}
        \centering
        \includegraphics[scale=.095]{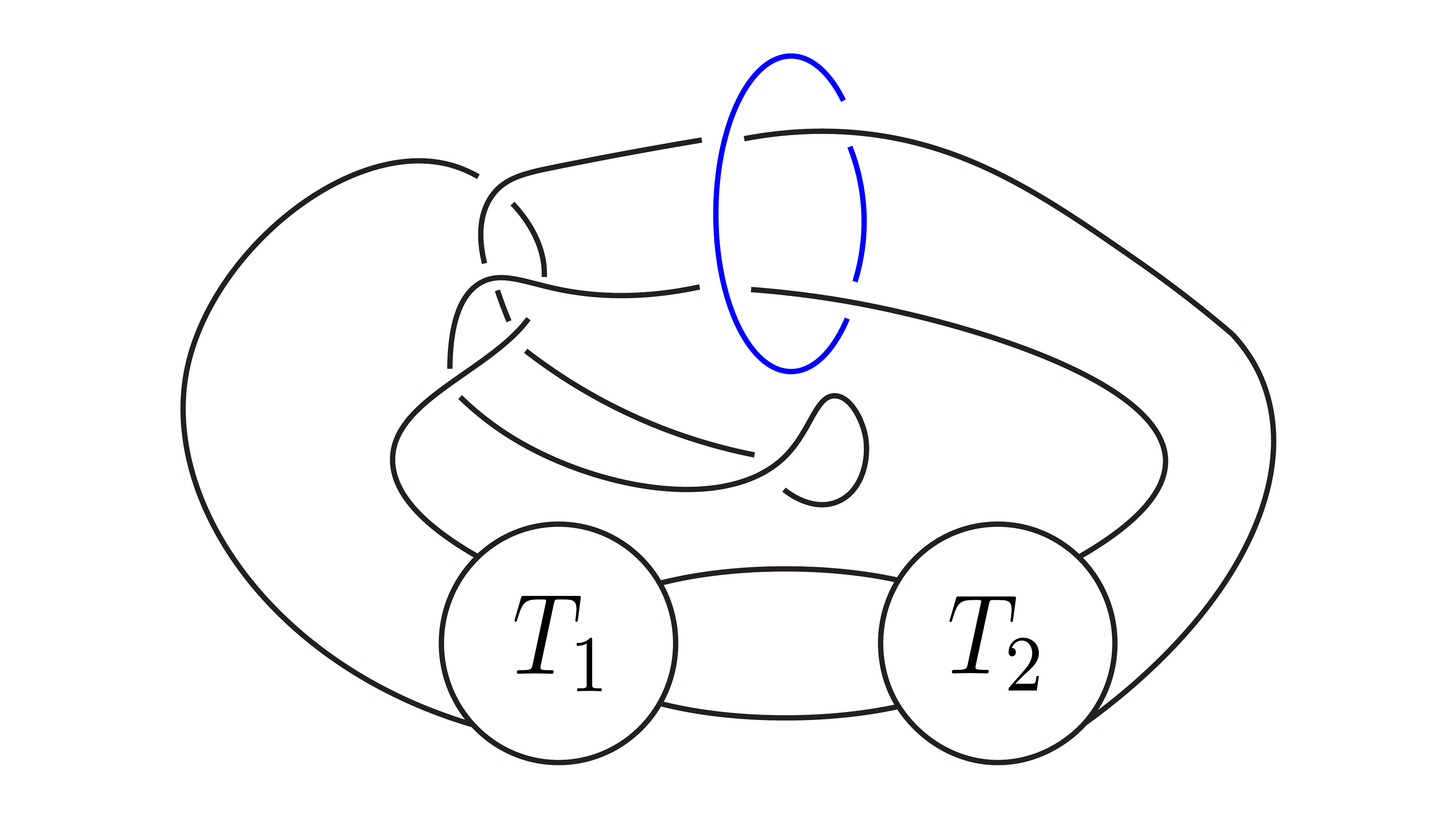}
        \caption{Another isotopy. }
        \label{fig:prob1_6_2}
    \end{minipage}%
     \begin{minipage}{.5\textwidth}
        \centering
      \includegraphics[scale=.095]{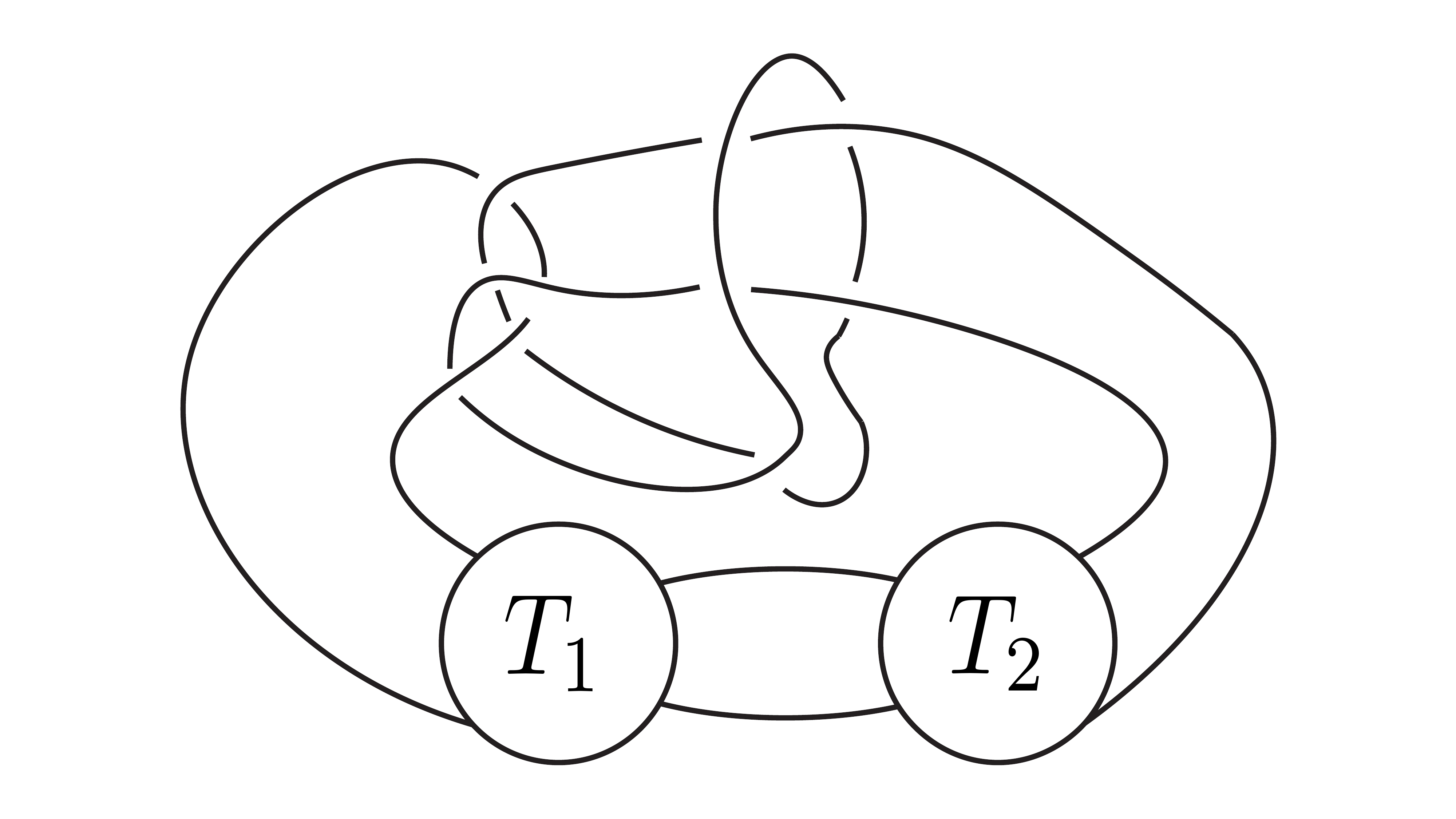}
        \caption{The result of adding a band; the final stage of the ribbon concordance between $K_1 \# K_2$  and the prime knot $(T_1 +_p T_2)+Cl$ .}
        \label{Final}
        \label{fig:prob1_6_2}
    \end{minipage}%

\end{figure}

\section{Proof of Theorem 1}\label{Proof}

In this section, we collect the results from the previous sections together to construct a proof of Theorem \ref{PrimeSq}. \\

\begin{proof}[Proof of Theorem 1]
By Lemma \ref{StandardPrime}, the knot $K$ from the proof of Lemma \ref{SQ} can be decomposed as a prime tangle $T_2$ and an untangle $T_1$, so that the denominator closure of $T_2$ is $K$ (this is ``K with ears"; see \cite{Bleiler}). Then, the knot $K_n=K \# \cdots \# K$ is the denominator closure of the (prime) tangle $T_2 +_p T_2 +_p \cdots +_p T_2$, where recall $+_p$ denotes the partial sum of tangles. Consider the ribbon concordance $C$ given in Theorem \ref{Prime}, from $K_n$ to $P_n:=(T_2 +_p T_2 +_p \cdots +_p T_2)+Cl$. This is illustrated, for $n=2$, in Figure \ref{Denominator}---Figure \ref{Final} by replacing $T_1$ by $T_2$ in the figures. By Theorem \ref{Sum} and Lemma \ref{Claspy}, $P_n$ is a prime knot.

By Theorem \ref{RibbonInj}, the map
$$F_C: \widetilde{\Kh}(K_n) \to \widetilde{\Kh}(P_n)$$ is injective with left inverse given by $F_{\overline{C}}$ where $\overline{C}$ is the concordance $C$ upside-down. Therefore $\widetilde{\Kh}(P_n)=\widetilde{\Kh}(K_n) \oplus G$ for some complement $G$. Theorem \ref{Cobordism} implies that the following commutes (note that the Euler characteristic of any concordance is $0$):x
\[
\begin{tikzcd}
\widetilde{\Kh}^{-n,0}(K_n) \arrow{r}{\Sq^n} \arrow[swap]{d}{F_C} & \widetilde{\Kh}^{0,0}(K_n) \arrow{d}{F_C} \\
\widetilde{\Kh}^{-n,0}(P_n)  \arrow{r}{\Sq^n} & \widetilde{\Kh}^{0,0}(P_n).
\end{tikzcd}
\]

This immediately implies Theorem \ref{PrimeSq}, since the vertical maps are injective. \qedhere\\

\end{proof}

Next, we prove Corollary \ref{PrimeSqUn} from the introduction.

\begin{proof}[Proof of Corollary 1]
This proof follows closely the proof of \cite[Corollary 1.4]{TylerLawsonRobertLipshitz}. There is a long exact sequence in Khovanov homology induced from the cofiber sequence:

$$\widetilde{\mathcal{X}}^{j-1}(P_n) \to \mathcal{X}^{j}(P_n) \to \widetilde{\mathcal{X}}^{j+1}(P_n). $$

The long exact sequence takes the form:

$$ \cdots \to \widetilde{\Kh}^{i,j+1}(P_n) \to \Kh^{i,j}(P_n) \xrightarrow{\pi} \widetilde{\Kh}^{i,j-1}(P_n) \to \widetilde{\Kh}^{i+1,j+1}(P_n) \to \cdots $$

Since over the field $\F:=\Z/2\Z$ the Khovanov homology of any knot $K$ is isomorphic to the direct sum $\widetilde{\Kh}^{i,j+1}(K;\F) \oplus \widetilde{\Kh}^{i,j-1}(K;\F)$ of the shifted reduced homology, the map $\pi$ above is surjective. So, there is a class $\gamma \in  \Kh^{-n,1}(P_n)$ so that $\pi(\gamma)=\beta$, where the class $\beta$ is as in the proof of Theorem \ref{PrimeSq}. Naturality of the Steenrod squares establishes the result. \qedhere

\end{proof}

\noindent\textit{Remark 1}: The above proof applies to any stable homotopy refinement of Khovanov homology that satisfies the analogue of Theorems \ref{connected} and \ref{Cobordism}. The idea of the proof also offers an obstruction to ribbon concordance between two knots. If $P$ and $Q$ are knots with a ribbon concordance between them, the Khovanov homology of $P$ is a summand of the Khovanov homology of $Q$ with the same stable cohomology operations as the Khovanov homology of $Q$. \\

As an illustration of the above remark, we have the following lemma. To state it, we recall that in \cite{2012arXiv1210.1882S}, Seed constructed pairs of links $L$ and $L'$ so that $\Kh(L;\Z) \cong \Kh(L';\Z)$ but the invariants $\mathcal{X}_{\Kh}(L)$ and $\mathcal{X}_{\Kh}(L')$ are not stably homotopy equivalent. Then, the following lemma is immediate. 

\begin{lemma}
For each of Seed's pairs of knots, there is no ribbon concordance between them.  

\end{lemma}

\section{Hyperbolic Knots and Invertible Concordances}\label{Hyp}

In this section we prove that there are hyperbolic knots with arbitrarily high Steenrod operations on their reduced Khovanov homology. The main theorem of this section, Theorem \ref{Hyper} from the introduction, is a direct consequence of the following that appears in \cite[Theorem 2.2(iv)]{Silver}:

\begin{theorem}

Given any knot $K \subset S^3$ there is a hyperbolic knot $H$ and a ribbon concordance from $K$ to $H$. 
\end{theorem}


The proof of Theorem \ref{Hyper} now is the same as the proof of Theorem \ref{PrimeSq} in Section \ref{Proof}. Following the notation of Section \ref{Proof}, and letting $C$ denote the composite ribbon concordance from $K_n$ to $H_n$, the following commutes:
\[
\begin{tikzcd}
\widetilde{\Kh}^{-n,0}(K_n) \arrow{r}{\Sq^n} \arrow[swap]{d}{F_C} & \widetilde{\Kh}^{0,0}(K_n) \arrow{d}{F_C} \\
\widetilde{\Kh}^{-n,0}(H_n)  \arrow{r}{\Sq^n} & \widetilde{\Kh}^{0,0}(H_n).
\end{tikzcd}
\]

\noindent\textit{Remark 2}: It was pointed out to us by Danny Ruberman that there is a stronger result possible. In Kawauchi \cite{kawauchi1989}, it is shown that for any knot $K$ there is an \emph{invertible} concordance from $K$ to a hyperbolic knot. This allows the propagation of Steenrod Squares without the injectivity results of Wilson or Levine-Zemke. See also \cite{Kim}.

\section{Satellite Knots}\label{SatelliteKnots}

In this section, we show how results from \cite{MR625818} imply Theorem \ref{Satellite}.

\begin{proof}[Proof of Theorem \ref{Satellite}]
The reader is referred to \cite{MR625818} for details of Livingston's construction. Glancing at Figure $2$ of \cite{MR625818} shows that there is a ribbon concordance from the unknot to a non-trivial knot $K'$ contained in the solid torus $S^1 \times D^2$. Now, consider the knot $K_n$ discussed in section \ref{Proof} and the satellite $S_n$ formed from $K'$ as pattern and $K_n$ as companion. Livingston shows that $S_n$ is prime. The ribbon concordance from the unknot to $K'$ gives a ribbon concordance from $K_n$ to $S_n$. The remainder of the proof goes through exactly as in the proof of Theorems \ref{PrimeSq} and \ref{Hyper}.
\end{proof}

\bibliography{Steenrod_on_prime_Revised}{}

\bibliographystyle{plain}

\subsection*{Holt Bodish}
 Department of Mathematics\\
 University of Oregon\\
hbodish@uoregon.edu

\end{document}